\pdfoutput=1
\RequirePackage{ifpdf}
\ifpdf % We~are running pdfTeX in pdf mode
\documentclass[pdftex]{sigma}
\else
\documentclass{sigma}
\fi

\newcommand{\R}{\mathbb{R}}

\newcommand{\N}{\mathbb{N}}

\newcommand{\E}{\mathcal{E}}
\DeclareMathOperator{\am}{am}
\DeclareMathOperator{\sn}{sn}

\DeclareMathOperator{\dn}{dn}

\numberwithin{equation}{section}

\newtheorem{Theorem}{Theorem}[section]
\newtheorem*{Theorem*}{Theorem}

\newtheorem{Lemma}[Theorem]{Lemma}

 { \theoremstyle{definition}

 }

\begin{document}
\allowdisplaybreaks

\renewcommand{\thefootnote}{}

\newcommand{\arXivNumber}{2306.12539}

\renewcommand{\PaperNumber}{021}

\FirstPageHeading

\ShortArticleName{On the Hill Discriminant of Lam\'e's Differential Equation}

\ArticleName{On the Hill Discriminant of Lam\'e's Differential\\ Equation\footnote{This paper is a~contribution to the Special Issue on Asymptotics and Applications of Special Functions in Memory of Richard Paris. The~full collection is available at \href{https://www.emis.de/journals/SIGMA/Paris.html}{https://www.emis.de/journals/SIGMA/Paris.html}}}

\Author{Hans VOLKMER}

\AuthorNameForHeading{H.~Volkmer}

\Address{Department of Mathematical Sciences, University of Wisconsin - Milwaukee, USA}
\Email{\href{mailto:volkmer@uwm.edu}{volkmer@uwm.edu}}

\ArticleDates{Received July 25, 2023, in final form March 08, 2024; Published online March 16, 2024}

\Abstract{Lam\'e's differential equation is a linear differential equation of the second order with a periodic coefficient involving the Jacobian elliptic function $\sn$ depending on the modulus $k$, and two additional parameters $h$ and $\nu$. This differential equation appears in several applications, for example, the motion of coupled particles in a periodic potential. Stability and existence of periodic solutions of Lam\'e's equations is determined by the value of its Hill discriminant $D(h,\nu,k)$. The Hill discriminant is compared to an explicitly known quantity including explicit error bounds. This result is derived from the observation that Lam\'e's equation with $k=1$ can be solved by hypergeometric functions because then the elliptic function $\sn$ reduces to the hyperbolic tangent function. A~connection relation between hypergeometric functions then allows the approximation of the Hill discriminant by a simple expression. In particular, one obtains an asymptotic approximation of $D(h,\nu,k)$ when the modulus $k$ tends to $1$.}


\Classification{33E10; 34D20}

\renewcommand{\thefootnote}{\arabic{footnote}}
\setcounter{footnote}{0}

\section{Introduction}%\label{intro}

Kim, Levi and Zhou \cite{Kim} consider two elastically coupled particles positioned at $x(t)$, $y(t)$ in a~periodic potential $V(x)$.
The system is described by
\[ \ddot{x}+V'(x)=\gamma (y-x),\qquad \ddot{y}+V'(y)=\gamma(x-y),\]
where $\gamma$ denotes the coupling constant.
Let $x(t)=y(t)=p(t)$ be a synchronous solution. If we linearize the system around
this synchronous solution, $x=p+\xi$, $y=p+\eta$, and set $u=\xi+\eta$, $w=\xi-\eta$, then we obtain
\begin{gather}
 \ddot{u}+V''(p)u=0,\nonumber\\%\label{1:eq1}
 \ddot{w}+(2\gamma+V''(p))w=0 .\label{1:eq2}
\end{gather}
These are Hill equations \cite{MW}, that is, they are of the form
\begin{equation}\label{1:eq3}
\ddot w+ q(t) w =0
\end{equation}
with a periodic coefficient function $q(t)$, say of period $\sigma>0$.
In this and many other applications the Hill discriminant $D$ associated with \eqref{1:eq3} plays an important role.
The discriminant $D$ is defined as the trace of the endomorphism $w(t)\mapsto w(t+\sigma)$ of the two-dimensional solution space of
\eqref{1:eq3} onto itself.
It is well known \cite{MW} that equation \eqref{1:eq3} is stable if $|D|<2$ and unstable if $|D|>2$. The condition $D=2$ is equivalent to
the existence of a nontrivial solution with period~$\sigma$ while $D=-2$ is equivalent to
the existence of a nontrivial solution with semi-period~$\sigma$.

In this work, we are interested in the special case $V(x)=-\cos x$. Then $p(t)$ is a solution of the differential equation $\ddot{p}+\sin p=0$ of the mathematical pendulum.
We get \cite[Section~22.19\,(i)]{dlmf}
\[
 p(t,\E)=2\am\left(\frac{t}{k},k\right),\qquad\text{where $k^2=\frac{2}{\E+2}$},
 \]
$\E$ denotes energy, and $\am$ is Jacobi's amplitude function \cite[Section~22.16\,(i)]{dlmf}.
Then equation~\eqref{1:eq2} becomes
\[
\frac{{\rm d}^2w}{{\rm d}t^2}+\left(2\gamma+1-2\sn^2\left(\frac{t}{k},k\right)\right) w=0,
\]
where $\sn(x,k)=\sin \am(x,k)$ is one of the Jacobian elliptic functions \cite[Section~16]{dlmf}.
If we substitute $t=ks$, we obtain Lam\'e's equation \cite[Section~29]{dlmf}
\begin{equation}\label{1:Lame}
\frac{{\rm d}^2w}{{\rm d}s^2} +\bigl(h-\nu(\nu+1)k^2\sn^2(s,k)\bigr) w=0
\end{equation}
with parameters $h=k^2(2\gamma+1)$ and $\nu=1$.
There is no explicit formula for the corresponding Hill discriminant
$D=D(h,\nu,k)$.
However, in \cite{Kim} a remarkable asymptotic formula for this Hill discriminant as $\E\to 0$ (or $k\to1$) is given.
It is shown that
\begin{equation}\label{1:eq4}
 D(h,1,k)=a\cos\left(\omega\ln \E-\phi\right)+o(\E) \qquad\text{as $\E\to 0$},
\end{equation}
where $\omega^2=2\gamma-1$.

The main result of this paper is Theorem \ref{t3}
which improves on \eqref{1:eq4} in three directions.
\begin{enumerate}\itemsep=0pt
\item
We allow any real $\nu$ in place of $\nu=1$. Since we may replace $\nu$ by $-1-\nu$ we assume $\nu\ge -\frac12$ without loss of generality.
\item
We provide explicit values for the amplitude $a$ and the phase shift $\phi$ in \eqref{1:eq4}
\item
We give explicit error bounds. This makes it possible to prove
stability of the Lam\'e equation in some cases.
\end{enumerate}

The idea behind the proof of Theorem \ref{t3} is the observation that Lam\'e's equation \eqref{1:Lame} with $k=1$
can be solved in terms of the hypergeometric function $F(a,b;c,x)$.
Then well-known connection relations between hypergeometric functions play a crucial role.

As a preparation, we present some elementary results on linear differential equations of the second order in
Section \ref{second}.
In Section~\ref{Lame1}, we give a quick review of the Lam\'e equation. In Section~\ref{Lame2}, we consider the special case
of the Lam\'e equation when $k=1$. In Section \ref{Hill} we combine our results to obtain Theorem \ref{t3}.

\section{Lemmas on second order linear equations}\label{second}

Let $u$ be the solution of the initial value problem
\[ u''+q(t)u=r(t),\qquad u(a)=u'(a)=0,\]
where $q,r\colon [a,b]\to\R$ are continuous functions.
By the variation of constants formula \cite[Section~2.6]{C},
\[ u(t)=\int_a^t L(t,s) r(s)\,{\rm d}s,\qquad u'(t)=\int_a^t \partial_1 L(t,s)r(s)\,{\rm d}s,\]
where $y(t)=L(t,s)$ is the solution of
\begin{equation}\label{ode1}
 y''+q(t)y=0
\end{equation}
determined by the initial conditions $y(s)=0$, $y'(s)=1$.
Let $L_1$, $L_2$ be constants such that
\begin{equation}\label{est}
 |L(t,s)|\le L_1,\quad |\partial_1 L(t,s)|\le L_2 \qquad\text{for $a\le s\le t\le b$}.
 \end{equation}
Then it follows that
\begin{equation}\label{est2}
\|u\|_\infty \le L_1 \int_a^b |r(s)|\,{\rm d}s,\qquad \|u'\|_\infty \le L_2 \int_a^b |r(s)|\,{\rm d}s,
\end{equation}
where $\|f\|_\infty:=\max_{t\in[a,b]} |f(t)|$.

\begin{Lemma}\label{l1}
Let $p,q\colon [a,b]\to \R$ be continuous. Let $L_1, L_2$ be as in \eqref{est}.
Let $y$ be a solution of \eqref{ode1} and $w$ a solution of $w''+p(t) w=0$ with $y(a)=w(a)$ and $y'(a)=w'(a)$.
Then
\begin{gather*}
 \|y-w\|_\infty \le L_1 \|w\|_\infty \int_a^b|p(s)-q(s)|\,{\rm d}s,\\
 \|y'-w'\|_\infty \le L_2 \|w\|_\infty \int_a^b|p(s)-q(s)|\,{\rm d}s.
\end{gather*}
\end{Lemma}
\begin{proof}
For $u=y-w$, we have
\[ u''(t)+q(t) u(t)=(p(t)-q(t))w(t) .\]
The desired result follows from \eqref{est2}.
\end{proof}

\begin{Lemma}\label{l2}
Let $q\colon [a,b]\to(0,\infty)$ be continuously differentiable and monotone. Set
 \[ m:=\min_{t\in[a,b]} q(t)>0,\qquad M:=\max_{t\in[a,b]} q(t).\]
Let $y_1$, $y_2$ be the solutions of \eqref{ode1} determined by $y_1(a)=y_2'(a)=1$, $y_1'(a)=y_2(a)=0$.
If $q$ is nondecreasing, then
\[ \|y_1\|^2_\infty\le 1,\qquad
 \|y_1'\|_\infty^2\le M,\qquad \|y_2\|_\infty^2\le \frac{1}{m}, \qquad \|y_2'\|_\infty^2\le \frac{M}{m},
\]
and, if $q$ is nonincreasing,
\[ \|y_1\|^2_\infty\le \frac{M}{m},\qquad
 \|y_1'\|_\infty^2\le M,\qquad \|y_2\|_\infty^2\le \frac{1}{m}, \qquad \|y_2'\|_\infty^2\le 1.
\]
\end{Lemma}
\begin{proof}
Suppose first that $q$ is nondecreasing. Set
\[ u_j(t):=y_j(t)^2+\frac{1}{q(t)}y_j'(t)^2 .\]
Then
\[ u_j'(t)=-\frac{q'(t)}{q(t)^2}y_j'(t)^2\le 0,\]
so $u_j(t)\le u_j(a)$ for all $t\in[a,b]$. Now $u_1(a)=1$ and $u_2(a)=\frac1m$ imply $y_1(t)^2\le 1$,
$y_1'(t)^2\le M$, $y_2(t)^2\le \frac1m$, $y_2'(t)^2\le \frac{M}{m}$.
If $q$ is nonincreasing, we argue similarly using $v_j(t)=y_j'(t)^2+q(t)y_j(t)^2$ in place of $u_j$.
\end{proof}

\section{Lam\'e's equation}\label{Lame1}

For $h\in \R$, $\nu\ge -\frac12$, $k\in(0,1)$, we consider the Lam\'e equation \cite[Section~IX]{A} and \cite[Section~XV]{EMO}%
\begin{equation}\label{ode2}
y''+\bigl(h-\nu(\nu+1)k^2\sn^2(t,k)\bigr) y=0 .
\end{equation}
This is a Hill equation with period $2K(k)$, where $K=K(k)$
is the complete elliptic integral of the first kind:
\[
K=\int_0^1\frac{{\rm d}t}{\sqrt{1-t^2}\sqrt{1-k^2t^2}}.
\]
Equation \eqref{ode2} also makes sense for $k=1$ \cite[Section~22.5\,(ii)]{dlmf} when it becomes
\begin{equation}\label{ode3}
y''+\bigl(h-\nu(\nu+1)\tanh^2 t\bigr) y=0 .
\end{equation}
Of course, this is not a Hill equation anymore.
Let
\[
y_1(t)=y_1(t,s,h,\nu,k) \qquad \text{and} \qquad y_2(t)=y_2(t,s,h,\nu,k)
\]
 be the solutions of \eqref{ode2}
determined by the initial conditions
\[
y_1(s)=y_2'(s)=1, \qquad y_1'(s)=y_2(s)=0.
\]

Set $q(t):=h-\nu(\nu+1)k^2\sn^2(t,k)$. This function is increasing on $[0,K]$ if $-\frac12\le \nu<0$ and decreasing
on $[0,K]$ if $\nu>0$.
We assume that $h>0$ and $h>\nu(\nu+1)k^2$. Then $q(t)>0$ for $t\in[0,K]$.
We define $H:=\bigl(h-\nu(\nu+1)k^2\bigr)^{1/2}$ and
\begin{alignat*}{3}%\label{C}
& C_1(h,\nu,k):=\begin{cases} 1 & \text{if $\nu<0$},\\h^{1/2}H^{-1} & \text{if $\nu\ge0$},
\end{cases}&&\qquad
C_1'(h,\nu,k):=\begin{cases} H & \text{if $\nu<0$},\\h^{1/2} & \text{if $\nu\ge0$},
\end{cases}& \\
& C_2(h,\nu,k):=\begin{cases} h^{-1/2} & \text{if $\nu<0$},\\H^{-1} & \text{if $\nu\ge0$},
\end{cases}&& \qquad
C_2'(h,\nu,k):=\begin{cases} h^{-1/2} H & \text{if $\nu<0$},\\1 & \text{if $\nu\ge0$}.
\end{cases}&
\end{alignat*}

\begin{Lemma}\label{l3}
Suppose that $h>0$ and $h-\nu(\nu+1)k^2>0$. Then, for $0\le s\le t\le K$,
\[
 |y_1(t,s)|\le C_1,\qquad
 |y_1'(t,s)|\le C_1',\qquad |y_2(t,s)|\le C_2, \qquad |y_2'(t,s)|\le C_2'.
\]
If $k=1$, this is true for all $0\le s\le t$.
\end{Lemma}
\begin{proof}
This follows from Lemma \ref{l2}.
\end{proof}

In the next theorem, we use the complete elliptic integral $E=E(k)$ of the second kind:
\[ E=\int_0^1 \frac{\sqrt{1-k^2t^2}}{\sqrt{1-t^2}}\,{\rm d}t .\]

\begin{Theorem}\label{t1}
Suppose that $h>0$ and $h-\nu(\nu+1)>0$. Then
\begin{gather*}
 |y_1(K,0,h,\nu,k)-y_1(K,0,h,\nu,1)|\le C_1C_2|\nu|(\nu+1)(E(k)-\tanh K(k)),\\
 |y_2'(K,0,h,\nu,k)-y_2'(K,0,h,\nu,1)|\le C_2C_2'|\nu|(\nu+1)(E(k)-\tanh K(k)),
\end{gather*}
where the constants $C$ are formed with $k=1$.
\end{Theorem}
\begin{proof}
We apply Lemma \ref{l1} with
\[ q(t)=h-\nu(\nu+1)k^2\sn^2(t,k),\qquad p(t)=h-\nu(\nu+1)\tanh^2 t,\]
and
\[ y(t)=y_1(t,0,h,\nu,k),\qquad w(t)=y_1(t,0,h,\nu,1) \]
on the interval $t\in[0,K]$. We note that \cite[formula~(4.4)]{V2004}
\[ k\sn(t,k)\le \tanh t\le \sn(t,k)\qquad \text{for $t\in[0,K]$} .\]
Therefore,
\begin{align*}
 \int_0^K|p(s)-q(s)|\,{\rm d}s&=|\nu|(\nu+1)\int_0^K \bigl(\tanh^2 s-k^2\sn^2(s,k)\bigr)\,{\rm d}s\\
 &= |\nu|(\nu+1)\int_0^K \bigl(\dn^2(s,k)-1+\tanh^2 s\bigr)\,{\rm d}s.
\end{align*}
Using $\int_0^K \dn^2(s,k)\,{\rm d}s= E$ \cite[p.~518]{WW}, we get
\[ \int_0^K|p(s)-q(s)|\,{\rm d}s=|\nu|(\nu+1)(E-\tanh K) .\]
By Lemma \ref{l3}, $|w(t)|\le C_1$ and we can choose $L_1=C_2$.
This gives the desired estimate for $y_1$. The estimate for $y_2'$ is proved similarly.
\end{proof}

Note that
\[ \int_0^K \bigl(\tanh^2s-k^2\sn^2(s,k)\bigr)\,{\rm d}s\le \bigl(1-k^2\bigr)\int_0^K \sn^2(s,k)\,{\rm d}s\le k'^2 K,\]
where $k'=\sqrt{1-k^2}$,
so
\[ E-\tanh K\le k'^2 K.\]
Also note that \cite[formula~(19.9.1)]{dlmf}
\[ K(k)\le \frac\pi2-\ln k' ,\]
so $E(k)-\tanh K(k)=O((1-k)\ln(1-k))$ as $k\to 1$.

\section[The Lam\'e equation for k=1]{The Lam\'e equation for $\boldsymbol{k=1}$}\label{Lame2}

Let $w_1$, $w_2$ be the solutions of \eqref{ode3} determined by initial conditions
$w_1(0)=w_2'(0)=1$, $w_1'(0)=w_2(0)=0$. Then $w_j(t)=y_j(t,0,h,\nu,1)$ using the notation of the previous section.
We assume that $\nu\ge -\frac12$ and $h>\nu(\nu+1)$, and set
\begin{equation}\label{omega}
 \mu:=\sqrt{\nu(\nu+1)-h}={\rm i}\omega,\qquad\text{where $\omega>0$}.
\end{equation}
The substitution $x=\tanh t$ transforms \eqref{ode3} to the associated Legendre equation \cite[formula~(14.2.2)]{dlmf}
of degree $\nu$ and order $\mu$. According to \cite[Section~5, formula~(15.09)]{O}, we
express~$w_j$ in terms of the hypergeometric function $F(a,b;c;z)$ as follows:
\begin{gather*}
\begin{split}
&w_1(t) = \cosh^\mu tF\bigl(-\tfrac12(\mu+\nu),\tfrac12(1-\mu+\nu);\tfrac12;\tanh^2 t\bigr),\\
&w_2(t) = \tanh t\cosh^\mu t F\bigl(\tfrac12(1-\mu-\nu),\tfrac12(2-\mu+\nu);\tfrac32;\tanh^2 t\bigr).
\end{split}
\end{gather*}
This can be confirmed by direct computation.
In order to determine the behaviour of the functions $w_j(t)$ as $\R\ni t\to\infty$, we use the connection formula \cite[formula~(15.8.4)]{dlmf} and find
$w_j(t)=\operatorname{Re}(v_j(t))$, where
\begin{gather*}
v_1(t) = \frac{A_1}{(2\cosh t)^{-\mu}}F\bigl(-\tfrac12(\mu+\nu),\tfrac12(1-\mu+\nu);1-\mu;\cosh^{-2} t\bigr),\\%\label{v1}
v_2(t) = \frac{A_2\tanh t}{(2\cosh t)^{-\mu}}F\bigl(\tfrac12(1-\mu-\nu),\tfrac12(2-\mu+\nu);1-\mu;\cosh^{-2} t\bigr),%\label{v2}
\end{gather*}
and
\begin{gather*}
A_1 = \frac{2^{1-\mu}\pi^{1/2}\Gamma(\mu)}{\Gamma\bigl(\tfrac12(1+\mu+\nu)\bigr)\Gamma\bigl(\frac12(\mu-\nu)\bigr)},\\
A_2 = \frac{2^{-\mu} \pi^{1/2}\Gamma(\mu)}{\Gamma\bigl(\tfrac12(2+\mu+\nu)\bigr)\Gamma\bigl(\frac12(1+\mu-\nu)\bigr)}.
\end{gather*}
We set
\[ z_j(t)=\operatorname{Re}\bigl(A_j{\rm e}^{{\rm i}\omega t}\bigr),\qquad j=1,2 .\]

\begin{Theorem}\label{t2}
Suppose $h>0$ and $h>\nu(\nu+1)$.
Then, for all $t\ge 0$,
\begin{gather*}
 |w_1(t)-z_1(t)|\le \omega^{-1}C_1|\nu|(\nu+1)(1-\tanh t),\\
 |w_2'(t)-z_2'(t)|\le C_2|\nu|(\nu+1)(1-\tanh t),
\end{gather*}
where $C_1$, $C_2$ are formed with $k=1$.
\end{Theorem}
\begin{proof}
Since $F(a,b;c;0)=1$, the representation $w_j(t)=\operatorname{Re}(v_j(t))$ yields
\begin{equation}\label{limit1}
\lim_{t\to\infty} w_j(t)-z_j(t)=\lim_{t\to\infty} \operatorname{Re}(A_j \bigl(2\cosh t)^{{\rm i}\omega}-A_j{\rm e}^{{\rm i}\omega t}\bigr) = 0.
\end{equation}
Similarly, we have
\begin{equation}\label{limit2}
 \lim_{t\to\infty} w_j'(t)-z_j'(t)=0 .
\end{equation}
The function $u_j=w_j-z_j$ satisfies
\[ u_j''+\omega^2 u_j=g_j(t), \qquad g_j(t):=\nu(\nu+1) \bigl(\tanh^2 t-1\bigr) w_j(t) .\]
Let $t_0, t\ge 0$. Then
\[ u_j(t)=u_j(t_0)\cos(\omega (t-t_0))+u_j'(t_0)\frac{\sin(\omega(t-t_0))}{\omega}+
\int_{t_0}^t \frac{\sin(\omega(t-s))}{\omega}g_j(s)\,{\rm d}s .\]
Letting $t_0\to\infty$, using \eqref{limit1}, \eqref{limit2} and Lemma \ref{l3}, we obtain
\[ |u_1(t)|\le\omega^{-1}\int_t^\infty |g_1(s)|\,{\rm d}s \le\omega^{-1}C_1|\nu|(\nu+1)(1-\tanh t) \]
as desired.
The estimate for $u_2'$ is derived similarly.
\end{proof}

The constant Wronskian of $z_1$, $z_2$ is
\[ z_1(t)z_2'(t)-z_1'(t)z_2(t) =\omega \operatorname{Im} \bigl(A_1\bar{A}_2\bigr) .\]
The reflection formula for the gamma function
\[ \Gamma(x)\Gamma(1-x)=\frac{\pi}{\sin(\pi x)}\]
gives
\begin{equation}\label{A1A2}
\omega A_1\bar{A}_2 = -\frac{\sin(\nu\pi)}{\sinh(\omega\pi)} +{\rm i}.
\end{equation}
Therefore,
\begin{equation*}%\label{wronskian}
 z_1(t)z_2'(t)-z_1'(t)z_2(t) =1.
\end{equation*}
Moreover,
\begin{equation}\label{zdisc}
 z_1(t)z_2'(t)+z_1'(t)z_2(t)=2z_1(t)z_2'(t)-1=\operatorname{Re}\bigl(B {\rm e}^{2{\rm i}\omega t}\bigr),
\end{equation}
where $B={\rm i}\omega A_1A_2$.
Using the duplication formula for the gamma function
\[ 2^{x-1}\Gamma\bigl(\tfrac12x\bigr)\Gamma\bigl(\tfrac12(x+1)\bigr)=\pi^{1/2} \Gamma(x)\]
we see that
\begin{equation}\label{B}
 B=\frac{\Gamma(1+\mu) \Gamma(\mu)}{\Gamma(1+\mu+\nu)\Gamma(\mu-\nu)} .
\end{equation}
If $\nu\in\N_0$, then
\[ B=\frac{({\rm i}\omega-1)({\rm i}\omega-2)\cdots({\rm i}\omega-\nu)}{({\rm i}\omega+1)({\rm i}\omega+2)\cdots({\rm i}\omega+\nu)}, \]
so $|B|=1$. If $\nu=1$, then
\[ B=\frac{{\rm i}\omega-1}{{\rm i}\omega+1}=\frac{\omega^2-1+{\rm i}2\omega}{\omega^2+1} \]
and
\[ \operatorname{Re}\bigl(B{\rm e}^{2{\rm i}\omega t}\bigr)=\frac{1}{\omega^2+1}\bigl(\bigl(\omega^2-1\bigr)\cos (2\omega t)-2\omega\sin(2\omega t)\bigr) .\]
By \eqref{A1A2},
\[ |B|^2=\big|\omega A_1 \bar{A}_2\big|^2=1+\frac{\sin^2\nu\pi}{\sinh^2\omega\pi}.\]
So $|B|>1$ if $\nu$ is not an integer.

\section{Hill's discriminant of Lam\'e's equation}\label{Hill}

The Hill discriminant $D(h,\nu,k)$ of Lam\'e's equation is given by \cite[p.~8]{MW}
\begin{equation}\label{hilldisc}
 D(h,\nu,k)=2(y_1(K)y_2'(K)+y_1'(K)y_2(K)) =2(2y_1(K)y_2'(K)-1),
 \end{equation}
where $y_j(t)=y_j(t,0,h,\nu,k)$ in the notation of Section \ref{Lame1}.
By combining Theorems \ref{t1} and \ref{t2}, we obtain the following main theorem of this work.

\begin{Theorem}\label{t3}
\qquad
\begin{enumerate} \itemsep=0pt
\item[$(a)$] Suppose $\nu\ge 0$ and $h>\nu(\nu+1)$. Then, for all $k\in(0,1)$,
\[ \bigl|D(h,\nu,k)-2\operatorname{Re}\bigl(B {\rm e}^{2{\rm i}\omega K(k)}\bigr)\bigr|\le 8h^{1/2}\omega^{-2}\nu(\nu+1)(E(k)+1-2\tanh K(k)).\]
\item[$(b)$] Suppose $\nu\in[-\frac12,0)$ and $h>0$. Then, for all $k\in(0,1)$,
\[\bigl|D(h,\nu,k)-2\operatorname{Re}\bigl(B {\rm e}^{2{\rm i}\omega K(k)}\bigr)\bigr|
 \le 8\omega h^{-1}|\nu|(\nu+1)(E(k)+1-2\tanh K(k)) .
\]
\end{enumerate}
The constants $\omega$ and $B$ are given in \eqref{omega} and \eqref{B}, respectively.
\end{Theorem}
\begin{proof}
Using \eqref{zdisc} and \eqref{hilldisc}, we have
\[ D(h,\nu,k)-2\operatorname{Re}\bigl(B {\rm e}^{2{\rm i}\omega K}\bigr)= 4(y_1(K)y_2'(K)-z_1(K)z_2'(K)) .\]
Using Lemma \ref{l3}, we estimate
\begin{gather*}
 \bigl|D(h,\nu,k)-2\operatorname{Re}\bigl(B {\rm e}^{2{\rm i}\omega K(k)}\bigr)\bigr|\\
 \qquad \le 4|y_1(K)||y_2'(K)-z_2'(K)|+4|z_2'(K)||y_1(K)-z_1(K)|\\
 \qquad \le 4C_1|y_2'(K)-z_2'(K)|+4C_2'|y_1(K)-z_1(K)|.
\end{gather*}
In fact, $|w_2'(t)|\le C_2'$
implies $|z_2'(t)|\le C_2'$ because of \eqref{limit2}.
Now we use Theorems \ref{t1} and \ref{t2} to estimate
\begin{gather*}
|y_1(K)-z_1(K)|\le |y_1(K)-w_1(K)|+|w_1(K)-z_1(K)|\\
\qquad \le C_1C_2|\nu|(\nu+1)(E-\tanh K)+\omega^{-1}C_1|\nu|(\nu+1)(1-\tanh K),
\end{gather*}
and
\begin{gather*}
|y_2'(K)-z_2'(K)|\le |y_2'(K)-w_2'(K)|+|w_2'(K)-z_2'(K)|\\
\qquad \le C_2C_2'|\nu|(\nu+1)(E-\tanh K)+C_2|\nu|(\nu+1)(1-\tanh K).
\end{gather*}
This gives the desired statements (a) and (b) substituting the values for $C_j$ and $C_j'$.
\end{proof}

We may use $K(k)=\ln(4/k')+O\bigl(k'^2\ln k'\bigr)$ \cite[formula~(19.12.1)]{dlmf} and $\bigl|{\rm e}^{{\rm i}s}-{\rm e}^{{\rm i}t}\bigr|\le |s-t|$ for $s,t\in\R$ to
obtain
\[ D(h,\nu,k)=2\operatorname{Re}\bigl(B {\rm e}^{2{\rm i}\omega \ln(4/k')}\bigr)+O((1-k)\ln(1-k))\qquad\text{as $k\to 1$}.\]

As an illustration, take $h=6$, $\nu=\frac12$ and $k=1-{\rm e}^{-\tau}$. Figure~\ref{fig1} depicts the graphs of $\tau\mapsto D\bigl(6,\frac12,k\bigr)$ (red) and $\tau\mapsto 2\operatorname{Re}\bigl(B {\rm e}^{2{\rm i}\omega K}\bigr)$ (black).
These graphs are hard to distinguish for $\tau>2$. The Hill discriminant $D\bigl(6,\frac12,k\bigr)$ is computed using~\eqref{hilldisc}. The values of $y_1(K)$ and $y_2'(K)$ are found by numerical integration
of Lam\'e's equation~\eqref{1:Lame} using the software \textsc{Maple}.

\begin{figure}[t] \centering
 \includegraphics[width=6cm]{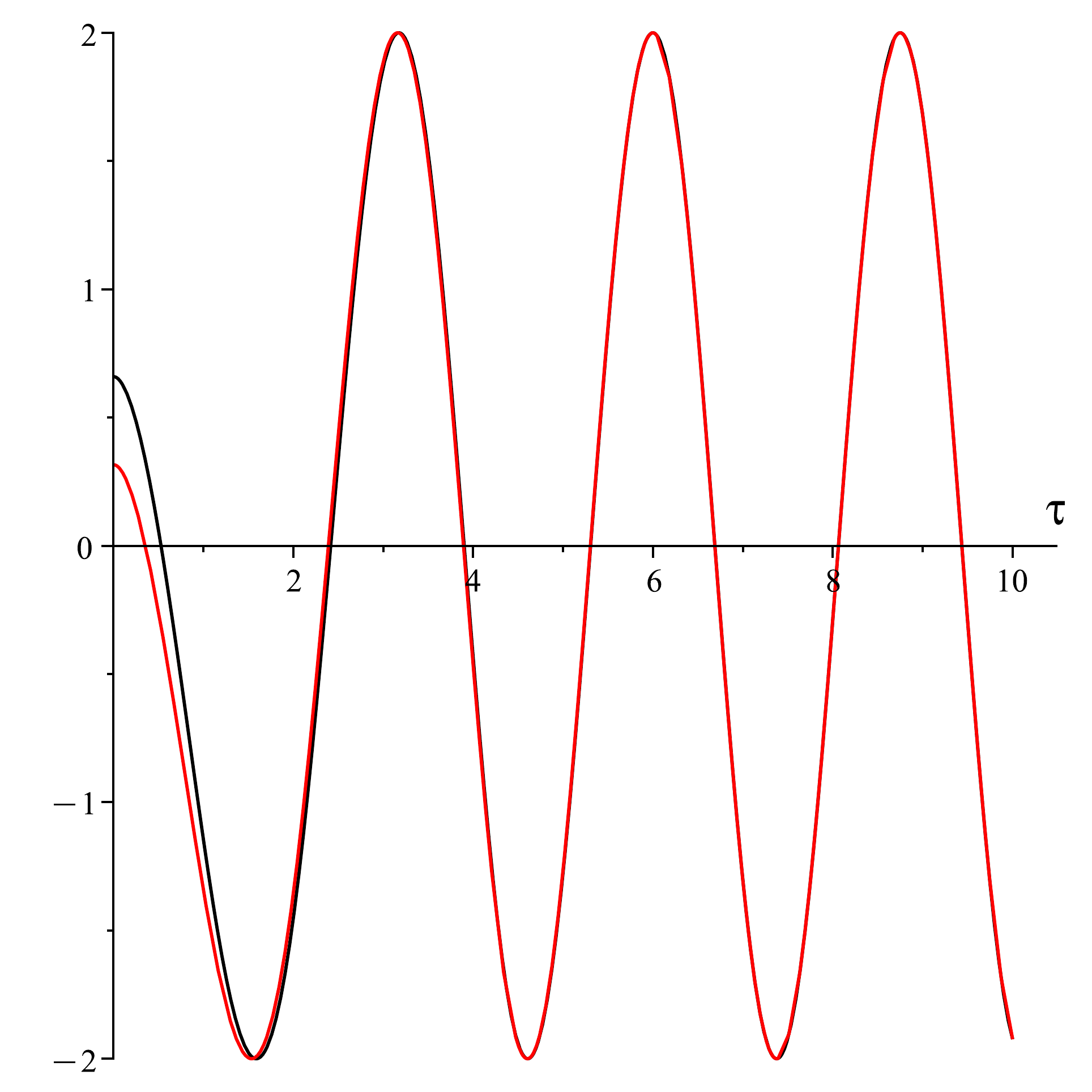}
\caption{Illustration to Theorem \ref{t3}. \label{fig1}}
\end{figure}

If $\tau=5$, then $k=0.993262\dots$ and $2\operatorname{Re}\bigl(B {\rm e}^{2{\rm i}\omega K}\bigr)=-1.274528\dots$. Theorem \ref{t3} gives $\bigl|D\bigl(6,\frac12,k\bigr)-2\operatorname{Re}\bigl(B {\rm e}^{2{\rm i}\omega K}\bigr)\bigr|\le 0.066641$.
Therefore, $|D(6,1,k)|<2$ and so Lam\'e's equation is stable for $h=6$, $\nu=\frac12$, $k=1-{\rm e}^{-5}$.

\section{Discussion and further work}

In Theorem \ref{t3}, we presented an asymptotic formula describing the behavior of the discriminant of the Lam\'e equation \eqref{1:Lame} as $k\to 1$. The proof is based on the fact that the Lam\'e equation approaches the
associated Legendre (a special case of the hypergeometric) differential equation, and the known
behavior of the hypergeometric function as the independent variable tends to $1$.
As we know from \cite{Kim} a less precise formula describing the asymptotic behavior as $\E\to0$ also exists for more general potentials in \eqref{1:eq2}. It is an interesting research question whether there exist other potentials that allow an
explicit determination of the amplitude and phase shift in this asymptotic formula.

\subsection*{Acknowledgements}

The author thanks the anonymous referees whose remarks led to an improvement of the paper.

\pdfbookmark[1]{References}{ref}
\LastPageEnding

\end{document}